\documentclass[journal, onecolumn]{IEEEtran}

\usepackage{amsmath,amssymb,amsfonts,amssymb}
\usepackage{mathtools}
\usepackage{bm}
\usepackage{booktabs,multirow,tabularx}
\usepackage{color, colortbl}

\usepackage{cite}
\DeclareMathAlphabet\mathbfcal{OMS}{cmsy}{b}{n}

\definecolor{Gray}{gray}{0.93}
\definecolor{LightCyan}{rgb}{1,1,1}

\def\tX{{\mathbfcal X}}
\def\tG{{\mathbfcal G}}

\def\tO{{\mathbfcal O}}
\def\mX{{\mathbf X}}
\newcommand{\eqdef}{:=}

\usepackage{braket,amsfonts}
\usepackage{array}

\usepackage{subcaption}

\usepackage{pgfplots}
\usepackage{amssymb, amsmath, amsthm, amsfonts}

\usepackage[ruled]{algorithm2e}

\newtheorem{theorem}{\bf Theorem} 
\newtheorem{lemma}{\bf Lemma}
\newtheorem{proposition}{\bf Proposition}
  
\ifCLASSINFOpdf
\else
\fi
\usepackage{setspace}

\begin{document}
 \doublespacing
%
\title{L1-norm Tucker Tensor Decomposition}
%
%
%

\author{ Dimitris G. Chachlakis,$^\dag$   Ashley Prater-Bennette,$^\ddag$  and Panos P. Markopoulos$^{\dag *}$ \thanks{$^\dag$D. G. Chachlakis and P. P. Markopoulos are with the Department of Electrical and Microelectronic Engineering, Rochester Institute of Technology, Rochester, NY ({dimitris@mail.rit.edu}, {panos@rit.edu}).}
\thanks{$^\ddag$A. Prater-Bennette is with the Air Force Research Laboratory, Information Directorate, Rome, NY ({ashley.prater.3@us.af.mil}).}
\thanks{$^*$Corresponding author.} 
\thanks{{\bf This is a preprint of an article that has been submitted for peer-reviewed publication. This preprint may contain errata. Some preliminary results were presented in} \cite{PPMGLOBALSIP2018}.}
\vspace{-1.5cm}
}

%
%

\markboth{L1-norm Tucker Tensor Decomposition (DRAFT)}%
{THIS IS A DRAFT}
%



\maketitle

\begin{abstract}
Tucker  decomposition is a common method for the analysis  of multi-way/tensor data.
 Standard Tucker has been shown to be sensitive against heavy corruptions, due to its  L2-norm-based formulation which places squared emphasis to peripheral entries.
In this work, we explore L1-Tucker, an L1-norm based reformulation of standard Tucker decomposition. 
After formulating the problem, we present  two  algorithms for its solution,  namely  L1-norm Higher-Order Singular Value Decomposition (L1-HOSVD) and L1-norm Higher-Order Orthogonal Iterations (L1-HOOI). The presented algorithms are accompanied by  complexity and convergence analysis.
Our numerical studies on tensor reconstruction and classification corroborate that  L1-Tucker, implemented by means of the proposed methods,   attains similar performance to standard Tucker when the processed data are corruption-free, while it exhibits sturdy resistance against heavily corrupted entries. 
\end{abstract}


%
\IEEEpeerreviewmaketitle

\section{Introduction}
\label{sec:intro}

Tucker decomposition  \cite{TUCKER,Kolda} is a cornerstone method for the analysis and compression of tensor data,  with a wide array of applications in data science \cite{papalex}, signal processing, machine learning \cite{SIDIROPOULOS},  and communications \cite{martin2014}, among other fields.  
Tucker decomposition is  typically computed by means of the Higher-Order Singular-Value Decomposition (HOSVD) algorithm, or  the Higher-Order Orthogonal Iterations (HOOI) algorithm \cite{Lathauwer,Kolda}. 
Other popular variants include Truncated HOSVD (T-HOSVD) \cite{Lathauwer,TUCKER,haardt2008higher}, Sequentially Truncated HOSVD (ST-HOSVD) \cite{vannie}, and Hierarchical HOSVD \cite{hierarchy}. 
Parallel algorithms for HOSVD have also been developed \cite{austin2016parallel}, with the ability to process very large  datasets.
When an $N$-way tensor is  processed as a collection of  $(N-1)$-way tensor measurements, Tucker is reformulated to Tucker2 decomposition \cite{ashley1}. For the special case of $N=3$, Tucker2 has also been presented as Generalized Low-Rank Approximation of Matrices (GLRAM) \cite{YE, sheehan2007higher} or 2-Dimensional Principal Component Analysis (2DPCA) \cite{2dpca, 2dpca2}.  For $N=2$,  both Tucker and Tucker2 boil down to standard matrix Principal-Component Analysis (PCA), which is practically solved through Singular-Value Decomposition (SVD) \cite{GOLUB}. 
In fact, both HOSVD and HOOI are high-order generalizations of SVD.

Due to its L2-norm formulation (minimization of the L2-norm of the residual-error or, equivalently, maximization of the L2-norm of the multi-way projection), standard Tucker decomposition has been shown to be sensitive against faulty entries within the processed tensor (also known as outliers), whether implemented by means of HOSVD, or HOOI\cite{FU, goldfarb, PPMSPL2018}. 
The same sensitivity has also been  documented in PCA, which is a special case of Tucker for $2$-way tensors (matrices). 
For matrix decomposition,  L1-norm-based PCA (L1-PCA) \cite{PPMTSP2014},  formulated by simple substitution of the L2-norm in PCA with the L1-norm, has exhibited solid robustness against heavily corrupted data in an array of applications \cite{PPMICIP2015, YLTM2016, PPMRADAR2017}.
Similar outlier resistance has been recently attained by algorithms for L1-norm reformulation of Tucker2 decomposition of $3$-way tensors (L1-Tucker2) \cite{PANG2010, DGCSPIE2018, DGCICASSP2018, PPMSPL2018}.

In this work we study {L1-Tucker}, an L1-norm reformulation of the general Tucker decomposition of $N$-way tensors. Then, we propose two new algorithms for the solution of L1-Tucker, namely L1-HOSVD and L1-HOOI, accompanied by formal convergence and complexity analysis. 
Our numerical studies show that the proposed L1-Tucker methods perform similar to standard Tucker methods (HOSVD and HOOI) when the processed data are nominal. However, L1-HOSVD and L1-HOOI  are markedly less affected by corruptions among the processed data.

\section{Technical Background}
\label{sec:background}

\subsection{Definitions and Notation}
\label{definitions}
An $N$-way tensor is an array of scalars, each entry of which is identified by $N$ indices. 
Vectors and matrices are $1$-way and $2$-way tensors, respectively. 
 An $N$-way tensor $\tX \in \mathbb R^{D_1 \times D_2 \times \ldots \times D_N}$ can also  be viewed as an $M$-way tensor in $\mathbb R^{D_1 \times D_2 \times \ldots \times D_M}$, for any $M>N$, with $D_m=1$ for $m >N$.
 For any fixed set of indices ${\{i_m\}_{m \in [N]\setminus n}}$,  vector 
$
\tX ({i_1, \ldots, i_{n-1}, : , i_{n+1}, \ldots, i_{N}}) \in \mathbb R^{D_n}
$
is called a mode-$n$ fiber  of $\tX$. 
Thus,  $\tX \in \mathbb R^{D_1 \times D_2 \times \ldots \times D_N}$  can also be viewed as a structured collection of its
$P_n := {\prod_{m \in [N]\setminus n} D_m}$  $n$-th mode fibers, for any ${n \in [N]:=\{ 1, 2, \ldots, N\}}$. 
Arranging all mode-$n$ fibers of tensor $\tX$ as columns of a matrix leads to a {mode-$n$ matrix unfolding} (also known as {flattening}) of $\tX$,  
${[\tX]_{(n)} \in \mathbb R^{D_n \times P_n}}$. 
Certainly, the mode-$n$ fibers of $\tX$ can be arranged in multiple different orders, resulting in column permutations of $[\tX]_{(n)}$. 
In this work, we consider the common unfolding order, by which tensor element $\tX ({i_1, i_2, \ldots, i_N})$ is mapped to the {mode-$n$} unfolding element $[\tX]_{(n)} ({i_n, j})$,  for ${j= 1 + \sum_{m \in [N]\setminus n}
(i_m -1)J_m}$ and 
${J_m := \prod_{
k \in [m-1]\setminus n} D_k}
$, for every $m \in [N]$~\cite{Kolda}.

\subsection{Tucker Decomposition}

Tucker tensor decomposition factorizes  $\tX$ into $N$ orthonormal bases and a core tensor. Specifically, considering $\{d_n\}_{n \in [N]}$ that satisfy ${d_n \leq D_n}$ ${\forall n \in [N]}$, Tucker decomposition is formulated as
\begin{align}
\underset{
\{\mathbf U_n \in \mathbb S(D_n, d_n)\}_{n \in [N]} 
}{\text{max.}}
\left\| \tX \textstyle  \times_{n \in [N]} \mathbf U_{n}^\top \right\|_F^2,
\label{Tucker}
\end{align}
where $\times_n$ denotes the mode-$n$ tensor-to-matrix product \cite{Kolda}, 
$\times_{n \in [N]} \mathbf U_{n}^\top $ summarizes the multi-mode product $ \times_{1} \mathbf U_{1}^\top \times_{2} \mathbf U_{2}^\top \times_{3} \ldots \times_{N} \mathbf U_{N}^\top$,   
$\mathbb S(D,d)=\{ \mathbf U \in \mathbb R^{D \times d};$ $~
\mathbf U^{\top} \mathbf U=\mathbf I_{d}\}$ is the Stiefel manifold containing all rank-$d$ orthonormal bases in $\mathbb R^{D}$,  and  $\| \cdot \|_F^2$  denotes the   L2 (or Frobenius) norm, returning  the summation of the squared entries of its tensor argument.
If  ${{\mathbf U}}^{\text{tckr}}_{n}$ is the mode-$n$  basis derived by solving \eqref{Tucker},  then 
\begin{align}
\tG  =  \tX  \times_{n \in [N]} {\mathbf U_{n}^{\text{tckr}}}^\top
\end{align} 
is the Tucker core of $\tX$, and $\tX$  is  Tucker-approximated by 
\begin{align}
\hat{\tX} =  \tG   \times_{n \in [N]} \mathbf U_n^{\text{tckr}}.
\end{align}
 Equivalently, $ \hat{\tX}  =   \tX  \times_{n \in [N]} \mathbf U_{n}^{\text{tckr}}{\mathbf U_n^{\text{tckr}}}^\top $.  If  $d_{n}=D_n~\forall n$,  it trivially holds that $\tX = \hat{\tX}$. The minimum values of $\{d_n\}_{n\in [N]}$ for which  $ \hat{\tX}  =   \tX $ are the respective mode ranks of $\tX$.  A schematic illustration of Tucker decomposition for $N=3$ is offered in Figure \ref{fig:schematic}. 
 
The solution to \eqref{Tucker} is commonly pursued by means of the Higher-Order Singular-Value Decomposition algorithm (HOSVD) \cite{Lathauwer} or the Higher-Order Orthogonal Iterations algorithm (HOOI) \cite{Kolda}.  A brief review of HOSVD and HOOI follows.

\begin{figure}[t!]
	
	\centering
	\includegraphics[width=0.6\textwidth]{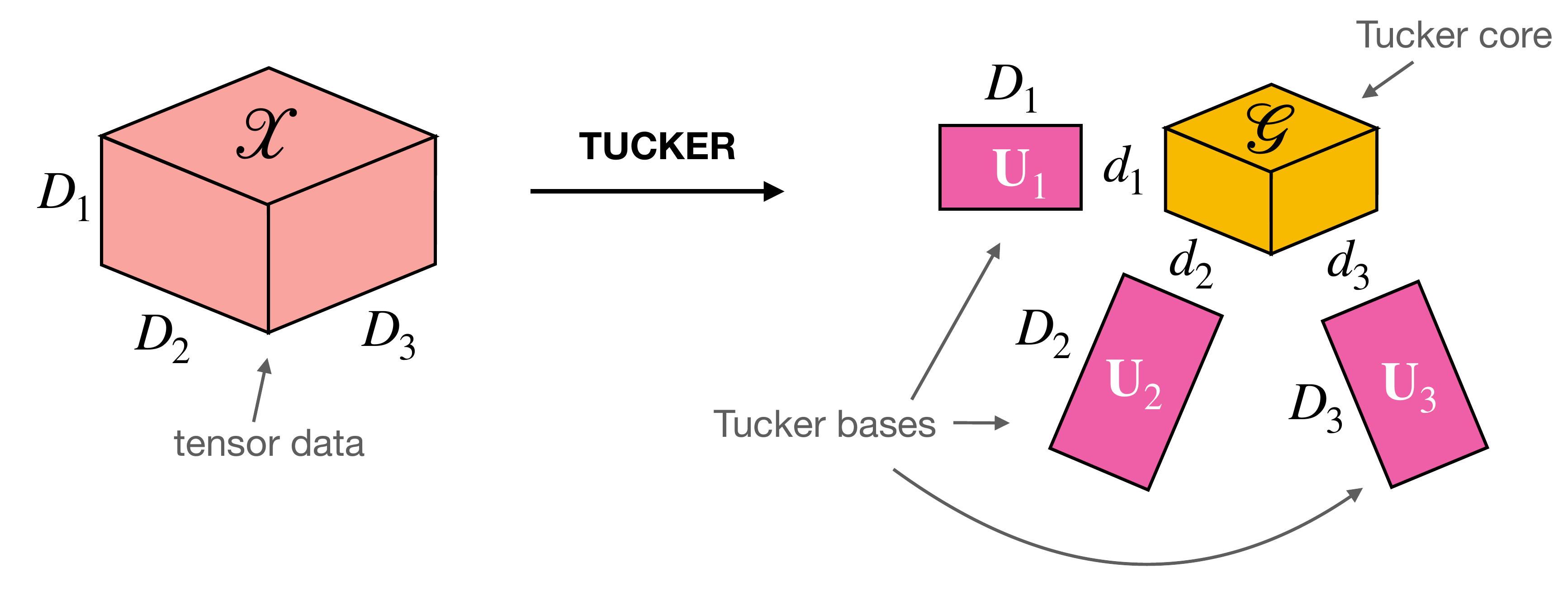}
	
	\caption{Schematic illustration of Tucker decomposition for $N=3$. }
	
	\label{fig:schematic}
\end{figure}

\subsubsection{HOSVD Method}

HOSVD approximates ${{\mathbf U}}^{\text{tckr}}_{n}$,  disjointly for each ${n\in [N]}$, by the $d_n$ principal components of the mode-$n$ unfolding $[\tX]_{(n)}$,  
\begin{align}
{\mathbf U}^{\text{hosvd}}_{n} \in
\underset{
\begin{smallmatrix}
\mathbf U \in \mathbb S(D_n, d_n)
\end{smallmatrix}
}{\text{argmax}}
~
\| \mathbf U^\top  [\tX]_{(n)}\|_F^2, 
\label{pca}
\end{align}
computed  by means of standard SVD.  Arguably, HOSVD draws motivation from the $N=2$ case, where $\tX=\mX$ is a $D_1 \times D_2$ matrix and  Tucker in \eqref{Tucker} simplifies to maximizing $ \| \mathbf U_1^\top \mathbf X \mathbf U_2 \|_F^2$. Indeed, for this special case of matrix  decomposition, the {optimal} orthonormal factors $\mathbf U_{n}^{\text{tckr}}$ 
can be  found disjointly and coincide with the solution to \eqref{pca}, $\mathbf U_{n}^{\text{hosvd}}$, where $[\tX]_{(1)}=\mathbf X$ and $[\tX]_{(2)}=\mathbf X^\top$. For the general $N
>2$ case, however, all $N$ bases are to  be found  jointly and the  HOSVD bases constitute, in general, approximations of the solution to \eqref{Tucker}.

\subsubsection{HOOI Method}

HOOI is a converging iterative procedure, which provably attains a higher value of the metric in \eqref{Tucker} than HOSVD, but still does not necessarily return the optimal solution \cite{sheehan2007higher, liu2014generalized, de2000best, xu2018convergence}.
For each mode $n\in [N]$, HOOI is typically (but not necessarily) initialized to  the solution of HOSVD as $\mathbf U_{n,0}^{\text{hooi}} = \mathbf U_n^{\text{hosvd}}$. Then,  the HOOI algorithm conducts a sequence of provably converging   iterations. 
That is, at the $t$-th iteration, HOOI sets 
 \begin{align}
 \mathbf A_{n,t}:=\left[\tX    \times_{m \in [n-1] } \mathbf U_{m, t}^{\text{hooi}}     \times_{k \in [N-n]+n}  \mathbf U_{k, t-1}^{\text{hooi}} \right]_{(n)}
 \end{align}
 and updates 
  \begin{align}
{\mathbf U}^{\text{hooi}}_{n,t} \in 
\underset{
\begin{smallmatrix}
\mathbf U \in \mathbb S(D_n, d_n)
\end{smallmatrix}
}{\text{argmax}}
~
 \| \mathbf U^\top  \mathbf  A_{n,t}\|_F^2,
 \label{pca2}
\end{align}
obtained again by standard SVD of $\mathbf  A_{n,t}$ ($d_n$ dominant left-singular vectors).

\subsection{Data Corruption and L1-norm Reformulation of PCA}
\label{sec_l1pca}

Large  datasets often contain heavily corrupted, outlying entries due to various causes, such as sensor malfunctions, errors in data storage/transfer, heavy-tail noise, intermittent variations of the sensing environment, and even intentional dataset ``poisoning'' \cite{bigdata}. 
Regretfully, such corruptions that lie far from the sought-after subspaces are known to  {significantly} affect PCA and its multi-way generalization, Tucker, even when they appear as a small fraction of the processed data \cite{PPMTSP2014, FU, PPMSPL2018}.  
Accordingly, in such cases,  the performance of any application that relies on PCA and Tucker can be compromised.
This corruption sensitivity can be largely attributed to the L2-norm formulation of Tucker/PCA, which places squared emphasis on each entry of the core,  thus  benefiting corrupted fibers. To demonstrate this  sensitivity of Tucker, we present the following simple example. 
We build  $\tX \in \mathbb R^{10 \times 10\times 10}$  with entries  independently drawn from $\mathcal N(0,1)$. Then, we corrupt additively the single entry $[\tX]_{2,3,4}$ with a point from $\mathcal N(0, \mu^2)$. We apply HOSVD on $\tX$ to obtain the single dimensional bases $\mathbf u_{1} \in \mathbb R^{10 \times 1}$, $\mathbf u_{2} \in \mathbb R^{10 \times 1}$, and $\mathbf u_{3} \in \mathbb R^{10 \times 1}$ and measure the aggregate normalized fitting of the bases to the corrupted fibers as 
$
f(\mu^2)=\sum_{i=1}^3 { |\mathbf u_i^\top \mathbf x_i|^2}{\| \mathbf x_{i}\|_2^{-2}},
$
 where $\mathbf x_{1}= [\tX]_{:,3,4}$, $\mathbf x_{2}= [\tX]_{2,:,4}$, and $\mathbf x_{2}= [\tX]_{2,3,:}$. We repeat this study $3000$ times and plot in Figure \ref{fig:outfit} the average value of $f(\mu^2)$, versus $\mu^2=0,10,\ldots, 100$. The impact of the single corrupted entry to the HOSVD bases (which tend to fit the corrupted fibers) is clearly documented. 

\begin{figure}[t!]
	
	\centering
	\includegraphics[width=0.7\textwidth]{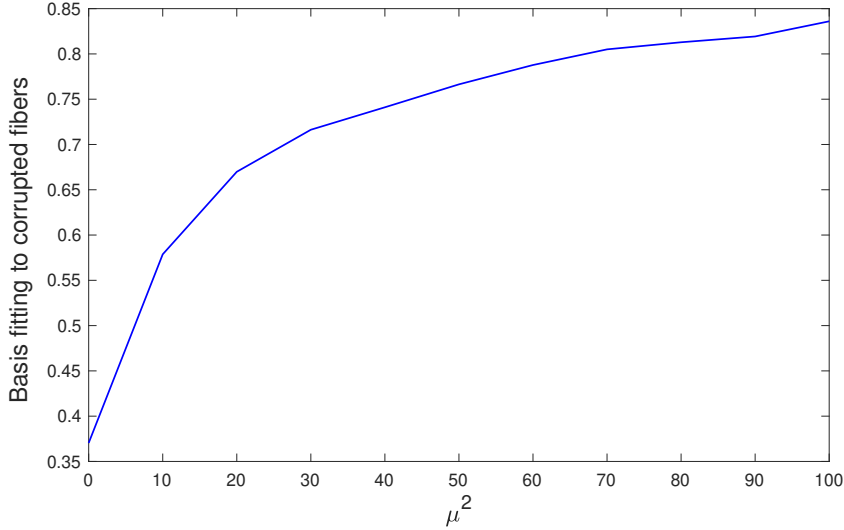}
	
	\caption{Fitting of HOSVD solution to corrupted fibers,  versus corruption variance $\mu^2$. }

	\label{fig:outfit}
\end{figure}

To counteract the effect of data contaminations/corruptions, researchers have resorted in ``robust'' reformulations of PCA and Tucker. 
One popular approach  seeks  to approximate the processed data tensor as the summation of a sought-after low-rank component and a jointly optimized sparse component that models outlier corruption  \cite{karamanis, lu2016tensor, goldfarb,li2015low}. This approach relies on ad hoc configured weights that regulate approximation rank,  sparsity, and iteration step size.
A rather more straightforward approach simply replaces the corruption-responsive L2-norm in PCA by the  L1-norm.
 In matrix analysis, this  modification resulted to the   L1-PCA formulation \cite{PPMTSP2014}, which constitutes a core component of the  L1-Tucker framework proposed in this work. Therefore, for completeness, we briefly present  below the theory behind L1-PCA, as well as a simple algorithm for its approximate computation.

Given a  data matrix $\mathbf X \in \mathbb R^{D_1  \times D_2}$  and $d_1\leq\text{rank}(\mathbf X)$, L1-PCA is defined as 
\begin{align}
\underset{
\begin{smallmatrix}
\mathbf U \in \mathbb S(D_1,d_1)
\end{smallmatrix}
}{\text{max.}}
~
\| \mathbf U^\top \mathbf X\|_1,
\label{l1pca}
\end{align}
where the L1-norm $\| \cdot \|_1$ returns the summation of the absolute entries of its matrix argument.
L1-PCA in \eqref{l1pca} was solved exactly  in \cite{PPMTSP2014}, where authors presented and leveraged the following Theorem \ref{l1pcaprop}.

\begin{theorem} 
Let $\mathbf B_{\text{\emph{nuc}}}$ be an optimal solution to 
\begin{align}
\underset{\mathbf B \in \{\pm 1\}^{D_2 \times d_1}}{\text{\emph{max.}}} \| \mathbf X \mathbf B\|_*.
\label{bnuc}
\end{align}
Then, $\mathbf U_{\text{\emph{L1}}} = \Phi(\mathbf X \mathbf B_{\text{\emph{nuc}}})$ is an optimal solution to L1-PCA in \eqref{l1pca}. Moreover, $\| \mathbf X^\top \mathbf U_{\text{\emph{L1}}}\|_1= \text{\emph{Tr}}\left(  \mathbf U_{\text{\emph{L1}}}^\top \mathbf X \mathbf B_{\text{\emph{nuc}}}\right) = \| \mathbf X \mathbf B_{\text{\emph{nuc}}}\|_*$  \emph{\cite{PPMTSP2014}}.
\label{l1pcaprop}
\end{theorem}

Nuclear norm $\| \cdot\|_*$  in \eqref{bnuc} returns the sum of the singular values of its matrix argument. For any tall matrix $\mathbf A \in \mathbb R^{D \times d}$ that admits SVD $\mathbf A = \mathbf W \mathbf S_{d \times d} \mathbf Q^\top$,   $\Phi(\cdot)$ in Theorem  \ref{l1pcaprop} is defined as $\Phi(\mathbf A) :=\mathbf W \mathbf Q^\top$. Moreover, by the Procrustes Theorem \cite{procrustes}, it holds that 
\begin{align}
\Phi(\mathbf A) \in \underset{\mathbf U \in \mathbb S(D,d)}{\text{argmax}}~{\text{Tr}(\mathbf U^\top \mathbf A)} = \underset{\mathbf U \mathbb \in \mathbb  S(D,d)}{\text{argmin}}~{ \| \mathbf U - \mathbf A \|_F}. 
\label{proc}
\end{align}

By means of the above Theorem \ref{l1pcaprop}, the solution to L1-PCA is obtained by the solution to \eqref{bnuc}, with an additional SVD step.
 \eqref{bnuc} can be solved by exhaustive search in its finite-size feasibility set, or more intelligent algorithms of lower cost, as shown in \cite{PPMTSP2014}. Computationally efficient,  approximate solvers  for \eqref{bnuc} and \eqref{l1pca} were  presented  in \cite{PPMTSP2017, arlington, kwak, tropp, savakis2019}.  Incremental solvers for L1-PCA were presented in \cite{panos2018, pados2016}.
 Algorithms for the complex-valued  L1-PCA were recently presented in \cite{complexL1, panos2018v2}. To date,  L1-PCA has found many applications in signal processing and machine learning, such as  radar-based  motion recognition and  foreground-activity extraction in  video sequences \cite{YLTM2016, PPMRADAR2017}. 
Most recently, L1-PCA was extended to L1-norm-based Tucker2 formulation, specifically for $3$-way tensors \cite{DGCSPIE2018, DGCICASSP2018}.

Next, we briefly present  the low-cost  L1-PCA calculator of \cite{arlington}, based on alternating optimization, which will be employed as a module of our subsequent L1-Tucker decomposition  developments.  A pseudocode for this L1-PCA solver is offered in Algorithm \ref{fig:l1pca}.
According to  \cite{PPMTSP2014}, 
\begin{align}
\underset{
\begin{smallmatrix}
\mathbf U \in \mathbb S(D_1 , d_1) 
\end{smallmatrix}
}{\text{max}}
~
\| \mathbf U^\top \mathbf X \|_1
 & =
\underset{
\begin{smallmatrix}
\mathbf U \in \mathbb S(D_1 , d_1 ) \\
\mathbf B \in \{ \pm 1\}^{D_2 \times d_1}
\end{smallmatrix}
}{\text{max}}
~
\text{Tr} \left( \mathbf U^\top  \mathbf X  \mathbf B \right) \\ 
& =
\underset{
\begin{smallmatrix}
\mathbf B \in \{ \pm 1\}^{D_2 \times d_1}
\end{smallmatrix}
}{\text{max}}
~
\| \mathbf X^\top \mathbf B \|_*
\label{st2}.
\end{align}
By \eqref{proc}, for  fixed $\mathbf B$,  the middle part of
 \eqref{st2} is maximized by $\mathbf U = \Phi(\mX \mathbf B)$. At the same time, for fixed  $\mathbf U$,  the middle part of \eqref{st2} is maximized by $\mathbf B = \text{sgn} (\mX^\top  \mathbf U )$, where  $\text{sgn}(\cdot)$ returns the $\pm 1$ signs of the entries of its argument ($\text{sgn}(0)=1$). By the above observations,  \cite{arlington} pursued a solution to \eqref{l1pca} in an alternating fashion, as 
 \begin{align}
 \mathbf B_{t} = \text{sgn}(\mX^\top \mathbf U_{t-1}) \in
  \underset{\mathbf B \in \{ \pm 1\}^{ D_2 \times d_1}}{\text{argmax}}   \text{Tr} \left( \mathbf U_{t-1}^\top  \mX \mathbf B \right) 
  \label{iter11}
   \end{align}
   and
    \begin{align}
\mathbf U_{t} 
=    \Phi(\mX \mathbf B_{t}) 
\in \underset{ 
\begin{smallmatrix}
\mathbf U \in \mathbb S(D_1 , d_1) 
\end{smallmatrix}
  }{\text{argmax}}  \text{Tr} \left( \mathbf U^\top \mX  \mathbf B_{t} \right),
   \label{iter21}
 \end{align}  
for $t=1, 2, \ldots$, and arbitrary initialization $\mathbf U_{0} \in \mathbb S({D, d})$. Omitting the explicit computation of the auxiliary matrix $\mathbf B_{t}$, \eqref{iter11}-\eqref{iter21} can take the  compact form
\begin{align}
\mathbf U_{t}= \Phi \left( \mX \text{sgn}( \mX^\top \mathbf U_{t-1}) \right).
\label{approxL1PCA}
\end{align}

\begin{algorithm}[t!]
 
		{ \texttt{L1PCA-AO}($\mathbf X$, $\mathbf U$)}
		
		{\hrule height 0.2mm}
		\vspace{0.05cm}
		\textbf{Input:} $\mathbf X \in \mathbb R^{D_1 \times D_2} $,  $\mathbf U \in \mathbb R^{D_1 \times d_1}$\\
		\begin{tabular}{r l l}
		1: & Until termination/convergence \\
		2: & ~~~~~~  $\mathbf B  \leftarrow  \text{sgn}(  \mathbf X^\top \mathbf U)  $		\\
		3: 	 & ~~~~~~  $\mathbf U  \leftarrow \Phi \left( \mathbf X \mathbf B \right)$
		\end{tabular} \\
		\textbf{Return:} $\mathbf U$ 
 
	\caption{\small{Alternating-optimization algorithm for the solution of L1-PCA \cite{arlington}. }
		}
 
	\label{fig:l1pca}
\end{algorithm}%

For completeness, we present below a proof of the monotonic metric  increase attained by the above iterations.
\begin{lemma}
For every $t \geq 1$,
\begin{align}
\|{\mathbf U_{t-1}}^\top \mX \|_1 
& = 
 \text{\emph{Tr}}({\mathbf U_{t-1}}^\top \mX \mathbf B_{t}) 
\\ & \leq 
  \text{\emph{Tr}}({\mathbf U_{t}}^\top \mX \mathbf B_{t}) 
\\ &  \leq
   \|{\mathbf U_{t}}^\top \mX \|_1.
   \label{key}
\end{align} 
\label{lemma1}
\end{lemma}
At the same time, the metric of \eqref{l1pca} is upper bounded by the exact L1-PCA solution  \cite{PPMTSP2014}. 
Thus, the  iteration in \eqref{approxL1PCA} is guaranteed to converge. 
In practice, iterations can be terminated when the metric-increase ratio 
\begin{align}
\frac{\|{\mathbf U_{t}}^\top  \mX \|_1 - \|{\mathbf U_{t-1}}^\top   \mX \|_1}{\|{\mathbf U_{t-1}}^\top  \mX \|_1}
\end{align}
 drops below a predetermined threshold $\tau>0$, or when $t$ exceeds a maximum number of permitted iterations. As shown in \cite{arlington},  the computational cost of the presented procedure is $\mathcal O(D_1 D_2 d_1 T)$, where $T$ is the number of iterations. In practice, $T$ can be set to be linear in $D_1$. Thus, the overall computational complexity of the above procedure is $\mathcal O(D_1^2 D_2 d_1)$.
In the sequel, we adopt the function notation $\texttt{L1PCA-AO}(\mathbf X, \mathbf U_{0})$ to denote the basis returned upon termination/converge of the alternating optimization of \eqref{approxL1PCA}.

\section{L1-Tucker Decomposition}

\subsection{Formulation}

Motivated by the corruption resistance of L1-PCA, in this work we study   {L1-Tucker} decomposition. L1-Tucker derives by simply replacing the L2-norm in \eqref{Tucker} by the  sturdier L1-norm, as
\begin{align}
\underset{
\{\mathbf U_n \in \mathbb S(D_n , d_n) \}_{n \in [N]} }{\text{max.}}
\left\| \tX \textstyle  \times_{n \in [N]} \mathbf U_{n}^\top \right\|_1.
\label{l1Tucker}
\end{align}
That is, \eqref{l1Tucker} strives to maximize  the sum of the absolute entries of the Tucker core --while standard Tucker maximizes the  sum of the squared entries of the core.
We note that, for any $m \in [N]$, 
\begin{align}
\left\| \tX \textstyle  \times_{n \in [N]} \mathbf U_{n}^\top \right\|_1 = 
\left\|  \mathbf U_{m}^\top \mathbf A_m \right\|_1,
\end{align}
where $\mathbf A_m = [\tX \times_{n < m} \mathbf U_n^\top  \times_{k >m} \mathbf U_k^\top ]_{(m)}$. Thus, with respect to each individual basis, the metric of L1-Tucker resembles that of L1-PCA.

For $N=3$,  $d_3=D_3$, and fixed $\mathbf U_3=\mathbf I_{D_3}$, L1-Tucker  in \eqref{l1Tucker} simplifies to a special L1-Tucker2 decomposition, proposed and studied in \cite{PANG2010, DGCICASSP2018}. For the even more special case of $d_1=d_2=1$, L1-Tucker2 was recently solved exactly in \cite{PPMSPL2018} through combinatorial optimization. For $N=2$ and fixed $\mathbf U_2=\mathbf I_{D_2}$, L1-Tucker simplifies to L1-PCA, in the form of \eqref{l1pca}. The above works have shown that, even in its special cases, the exact solution to L1-Tucker is hard to find --which also holds true for standard Tucker decomposition.
 
Next, we present the first two approximate algorithms for solving the general L1-Tucker decomposition, in the form of \eqref{l1Tucker}.


\subsection{Proposed L1-HOSVD Method}

We first present L1-HOSVD, an algorithm analogous to HOSVD for standard Tucker. 
Specifically, for every  $n \in [N]$,  L1-HOSVD seeks to optimize the mode-$n$ basis $\mathbf U_{n}$ in \eqref{l1Tucker} individually, by  L1-PCA solution (exact or approximate) of the mode-$n$ matrix unfolding of $\tX$, $[\tX]_{(n)}$. That is, 
L1-HOSVD approximates the mode-$n$ basis in the solution of  \eqref{l1Tucker} by   
\begin{align}
\mathbf U_{n}^{\text{l1-hosvd}} \in
\underset{
\begin{smallmatrix}
\mathbf U \in \mathbb S(D_n , d_n) 
\end{smallmatrix}
}{\text{argmax}}
~
\left\| \mathbf U^\top [\tX]_{(n)} \right\|_1.
\label{l1pca2}
\end{align}
Similar to HOSVD, L1-HOSVD  decouples  the basis optimization task across the modes.
 We observe that \eqref{l1pca2} is, in fact, L1-PCA of the mode-$n$ flattening of $\tX$. 
 As mentioned above, there are multiple algorithms in the literature for solving L1-PCA in \eqref{l1pca2}, both exactly and approximately \cite{PPMTSP2014, PPMTSP2017, arlington}. 
 As an L1-Tucker solution framework,  L1-HOSVD allows for the employment of any solver for \eqref{l1pca2}, thus making possible different performance/cost trade-offs.
In this work,  we demonstrate the employment of  the simple alternating-optimization method of Algorithm  \ref{fig:l1pca},  initialized, for example, at the solution of standard HOSVD. That is, L1-HOSVD returns
\begin{align}
{\mathbf U}^{\text{l1-hosvd}}_{n} = \texttt{L1PCA-AO} ([\tX]_{(n)}, \mathbf U^{\text{hosvd}}_{n} ),
\label{l1pcatuck}
\end{align}
for every $n \in [N]$.  As presented in Section \ref{sec_l1pca},  the computational cost of \eqref{l1pcatuck} is  $\mathcal O(D_n d_n P)$, where $P:= \prod_{m \in [N]}D_m$, when the number of L1-PCA iterations is linear in $D_n$.
A pseudocode of   L1-HOSVD  is offered in Algorithm \ref{algo}.

\begin{algorithm}[t!]
 
		{\texttt{L1HOSVD}($\tX$, $\{d_{n}\}_{n \in [N]}$)}
		
		{\hrule height 0.2mm}
		\vspace{0.05cm}
		\textbf{Input:} $\tX \in \mathbb R^{D_1 \times \ldots \times D_N},  \{d_{n}\}_{n \in [N]} $\\
		\begin{tabular}{r l l}
		1: & Initialize $\{ \mathbf U_{n}  \}_{n \in [N]}$ by HOSVD of $\tX$ for   $\{d_{n}\}_{n \in [N]} $ \\
		2: & For $n=1, 2, \ldots, N$ \\
		3: & ~~~~~~ $\mathbf U_n \leftarrow$ \texttt{L1PCA-AO}($[\tX]_{(n)}$, $\mathbf U_n$) 
		\end{tabular} \\
		\textbf{Return:} $\{ \mathbf U_{n}  \}_{n \in [N]}$ 
 
	\caption{\small{The proposed L1-HOSVD algorithm for solving  L1-Tucker in \eqref{l1Tucker}.}
	}
 
	\label{algo}
\end{algorithm}%

\subsection{Proposed L1-HOOI Method}

Similar to HOSVD, for general dense processed tensors,  the disjoint mode optimization  of L1-HOSVD may be very limiting. Next,   we present  {L1-HOOI}, an iterative algorithm for approximating the solution to   \eqref{l1Tucker}. Specifically, L1-HOOI is arbitrarily initialized to   $N$ feasible bases and conducts a sequence of iterations in   which  it updates all bases $\{\mathbf{U}_n \}_{n \in [N]}$ such that the objective value of L1-Tucker  increases. Thus, when initialized to the L1-HOSVD bases, L1-HOOI is guaranteed to outperform L1-HOSVD in the L1-Tucker metric. In this approach, L1-HOSVD can be viewed as an initialization for L1-HOOI, or, conversely, L1-HOOI can be viewed as a refinement of L1-HOSVD.

Formally, L1-HOOI first initializes bases $\{ \mathbf U_n^{(0)} \in \mathbb S({D_n, d_n})\}_{n \in [N]}$  --for example, $ \mathbf U_n^{(0)} = \mathbf U_n^{\text{l1-hosvd}}$. 
Then, at the $q$-th iteration, $q=1,2,\ldots$,  it successively optimizes  all $N$  bases,  in increasing mode order $n=1,2,\ldots$. Specifically,  at a given iteration $q$ and mode index $n$, L1-HOOI   fixes $\mathbf U_{m}^{(q)}$ for $m<n$ and $\mathbf U_{k}^{(q-1)}$ for $k>n$ and seeks the mode-$n$ basis $\mathbf U_{n}^{(q)}$ that maximizes the L1-Tucker metric. 
That is, for a given index pair $(q,n)$, 
L1-HOOI pursues    
 \begin{align}
  \mathbf U_n^{(q)} \in
\underset{
 \mathbf U \in \mathbb S(D_n , d_n)}
 {\text{argmax}}
\left\| \tX   \times_{m <n} {\mathbf U_{m}^{(q)}}^\top \times_n \mathbf U^\top  \times_{k>n} {\mathbf U_{k}^{(q-1)}}^\top \right\|_1.
\label{L1HOOIupdt}
\end{align}
Defining
 \begin{align}
\mathbf A_n^{(q)}\eqdef \left[\tX   \times_{m < n}  {\mathbf U_m^{(q)}}^\top    \times_{k > n} {\mathbf U_k^{(q-1)}}^\top \right]_{(n)},
\end{align}
\eqref{L1HOOIupdt} is equivalently rewritten in the familiar L1-PCA form
 \begin{align}
  \mathbf U_n^{(q)} \in 
  \underset{
 \mathbf U \in \mathbb S(D_n , d_n) }{\text{argmax}}{\left\| {\mathbf U}^\top \mathbf A_n^{(q)}\right\|_1}.
\label{l1hooi2}
\end{align}
We notice that, in contrast to \eqref{l1pca2},  the metric of \eqref{l1hooi2}  involves the jointly optimized bases of the other modes. 
As discussed above, there are multiple solvers  for \eqref{l1hooi2} that can attain different performance/cost trade-offs. The proposed L1-HOOI framework can be combined with any L1-PCA solver.  In the sequel, we employ again the iterative solver of  Algorithm  \ref{fig:l1pca}.  That is,  for any   ($q,n$), we  set  
\begin{align}
\mathbf U_{n}^{(q)} = \texttt{L1PCA-AO}(\mathbf A_{n}^{(q)}, ~ \mathbf U_{n}^{(q-1)}).
\label{update1}
\end{align}
A pseudocode of the proposed L1-HOOI method is offered in Algorithm \ref{algo2}. 
A formal convergence analysis of the L1-HOOI iterations is presented below. 

\begin{algorithm}[t!]
	 
		{\texttt{L1HOOI}($\tX$, $\{d_{n}\}_{n \in [N]}$)}
		
		{\hrule height 0.2mm}
		\vspace{0.06cm}	
		\textbf{Input:} $\tX \in \mathbb R^{D_1 \times \ldots \times D_N},  \{d_{n}\}_{n \in [N]} $\\
		\begin{tabular}{r l l}
			1: & Initialize $\{ \mathbf U_{n} \}_{n \in [N]} \leftarrow$ \texttt{L1HOSVD}($\tX$,  $\{d_{n}\}_{n \in [N]}$) \\
		    2: & Until termination/convergence \\
			3: & ~~~~~~ For $n=1, 2, \ldots, N$ \\
			4: & ~~~~~~~~~~~~~~~ $\mathbf A \leftarrow [\tX   \times_{m < n} {\mathbf U_m}^\top    \times_{k > n}  {\mathbf U_k}^\top ]_{(n)}$ \\
			5: & ~~~~~~~~~~~~~~~ 	$\mathbf U_n \leftarrow$ \texttt{L1PCA-AO}($\mathbf A$, $\mathbf U_n$) 
		\end{tabular} \\
		\textbf{Return:} $\{ \mathbf U_{n} \}_{n \in [N]}$ 
	\caption{\small{The proposed L1-HOOI algorithm for solving  L1-Tucker in \eqref{l1Tucker}.}
	}
	\label{algo2}
\end{algorithm}%

\subsubsection{Convergence}
\label{conv}
 

We start with introducing Lemma \ref{remark1}, which shows that the $q$-th update of the mode-$n$ basis increases the L1-Tucker metric.

\begin{lemma}
Lemma \ref{lemma1} implies that, for fixed $\{\mathbf U_m^{(q)}\}_{m<n}$ and $\{\mathbf U_k^{(q-1)}\}_{k>n}$,  
\begin{align}
 \left\|{\mathbf U_{n}^{(q)}}^\top\mathbf A_n^{(q)}\right\|_1  
   \geq   \left\|{\mathbf U_n^{(q-1)}}^\top\mathbf A_n^{(q)}\right\|_1.
\end{align}
  \label{remark1}
\end{lemma}

We note that  Lemma \ref{remark1} would also hold if, instead of   \eqref{update1}, we employed the bit-flipping iterative L1-PCA solver of \cite{PPMTSP2017}, initialized  to ${\mathbf U_n^{(q-1)}}$. Also, Lemma \ref{remark1} certainly holds true if ${\mathbf U_n^{(q)}}$ is computed by the exact solution of  \eqref{l1hooi2} --that is, by means of the exact algorithms of \cite{PPMTSP2014}. The following new Lemma \ref{lemma:hooi1} shows that, within the same iteration, the metric increases as we successively optimize the bases.

\begin{lemma}
	For any   $q > 0$ and every $n>m \in [N]$, it holds that
	\begin{align}
	\left\|{\mathbf U_n^{(q)}}^\top \mathbf A_n^{(q)}\right\|_1 \geq 	\left\| {\mathbf U_m^{(q)}}^\top \mathbf A_m^{(q)}\right\|_1.
	\end{align}
	\label{lemma:hooi1}
\end{lemma}
\begin{proof}
It holds that
	\begin{align}
	\left\|{\mathbf U_n^{(q)}}^\top \mathbf A_n^{(q)}\right\|_1	&=  \left\| {\mathbf U_n^{(q)}}^\top \left[\tX \times_{k<n} {\mathbf U_k^{(q)}}^\top \times_{l>n} {\mathbf U_l^{(q-1)}}^\top \right]_{(n)} \right\|_1
	\\& \overset{\text{Lemma } \ref{remark1}}{\geq}\left\| {\mathbf U_n^{(q-1)}}^\top \left[\tX \times_{k<n} {\mathbf U_k^{(q)}}^\top \times_{l>n} {\mathbf U_l^{(q-1)}}^\top \right]_{(n)} \right\|_1
	\\&= \left\| {\mathbf U_{n-1}^{(q)}}^\top \left[\tX \times_{k<n-1} {\mathbf U_k^{(q)}}^\top \times_n {\mathbf U_n^{(q-1)}}^\top \times_{l>n} {\mathbf U_l^{(q-1)}}^\top \right]_{(n-1)} \right\|_1
	\\&=\left\| {\mathbf U_{n-1}^{(q)}}^\top \mathbf A_{n-1}^{(q)}\right\|_1.
	\end{align}
	By induction, for every $n>m$, $\left\|{\mathbf U_n^{(q)}}^\top \mathbf A_n^{(q)}\right\|_1 \geq 	\left\| {\mathbf U_m^{(q)}}^\top \mathbf A_m^{(q)}\right\|_1$.
\end{proof}

The following new Lemma \ref{lemma:hooi2} and Proposition \ref{hooiprop} conclude our analysis on the monotonic increase of the L1-Tucker metric across the iterations of L1-HOOI.

\begin{lemma}
	For every $q>0$, it holds that 
	\begin{align}
    	\left\|{\mathbf U_1^{(q)}}^\top \mathbf A_1^{(q)}\right\|_1 \geq 	\left\| {\mathbf U_N^{(q-1)}}^\top \mathbf A_N^{(q-1)}\right\|_1.
	\end{align}
		\label{lemma:hooi2}
\end{lemma}
\begin{proof}
It holds that
	\begin{align}
		\left\|{\mathbf U_1^{(q)}}^\top \mathbf A_1^{(q)}\right\|_1	&= \left\| {\mathbf U_1^{(q)}}^\top \left[\tX \times_{k>1} {\mathbf U_k^{(q-1)}}^\top \right]_{(1)} \right\|_1
		\\& \overset{\text{Lemma } \ref{remark1}}{\geq}\left\| {\mathbf U_1^{(q-1)}}^\top \left[\tX \times_{k>1} {\mathbf U_k^{(q-1)}}^\top \right]_{(1)} \right\|_1
		\\&= \left\| {\mathbf U_N^{(q-1)}}^\top \left[\tX \times_{k<N} {\mathbf U_k^{(q-1)}}^\top \right]_{(N)} \right\|_1
		\\&=\left\| {\mathbf U_N^{(q-1)}}^\top \mathbf A_N^{(q-1)}\right\|_1.
	\end{align}
\end{proof}
In view of Lemmas \ref{remark1}, \ref{lemma:hooi1}, and \ref{lemma:hooi2}, the following Proposition \ref{hooiprop} holds true. 

\begin{proposition}
	For any $n  \in [N]$ and every $q' > q$ 
	\begin{align}
	\left\|{\mathbf U_{n}^{(q')}}^\top\mathbf A_{n}^{(q')}\right\|_1 \geq \left\|{\mathbf U_n^{(q)}}^\top\mathbf A_n^{(q)}\right\|_1. 
	\end{align}
	\label{hooiprop}
\end{proposition}
\begin{proof}
It holds that
\begin{align}
\left\|{\mathbf U_{n}^{(q')}}^\top\mathbf A_{n}^{(q')}\right\|_1 
& \overset{\text{Lemma } \ref{lemma:hooi1}} {\geq}
\left\|{\mathbf U_{1}^{(q')}}^\top\mathbf A_{1}^{(q')}\right\|_1  \\
& \overset{\text{Lemma } \ref{lemma:hooi2}} {\geq}
\left\|{\mathbf U_{N}^{(q'-1)}}^\top\mathbf A_{N}^{(q'-1)}\right\|_1  \\
& \overset{\text{Lemma } \ref{lemma:hooi1}} {\geq}
\left\|{\mathbf U_{1}^{(q'-1)}}^\top\mathbf A_{1}^{(q'-1)}\right\|_1 \\
& \overset{~~~~~~~~~~~} {\geq}
\left\|{\mathbf U_{N}^{(q)}}^\top\mathbf A_{N}^{(q)}\right\|_1  \\
& \overset{~~~~~~~~~~~} {\geq}
\left\|{\mathbf U_{n}^{(q)}}^\top\mathbf A_{n}^{(q)}\right\|_1.
\end{align}
\end{proof}

Denoting $p:=\prod_{n \in [N]} d_n$, the following Lemma \ref{lemineq} provides an upper bound for the L1-Tucker metric. 

\begin{lemma}
For any $\{\mathbf U_n \in \mathbb S(D_n, d_n) \}_{n\in [N]}$, it holds that 
\begin{align}
\| \tX \times_{n \in [N]} \mathbf U_n^\top \|_1 \leq \sqrt{p} \| \tX \|_F.
\end{align}
\label{lemineq} 
\end{lemma}

\begin{proof}
Let $\mathbfcal Y = \tX \times_{n \in [N]} \mathbf U_n^\top \in \mathbb R^{d_1 \times \ldots \times d_N}$ and define $\mathbf y = \text{vec}([\mathbfcal Y]_{(1)})$ and $\mathbf x= \text{vec}([\mathbfcal X]_{(1)})$. It holds that $\| \mathbfcal Y \|_1 = \| \mathbf y\|_1$ and  $\| \tX \|_F = \| \mathbf x\|_2$. Also, define $\mathbf Z = \mathbf U_N \otimes  \mathbf U_{N-1} \otimes \ldots \otimes \mathbf U_{1} \in \mathbb S(P, p)$. We observe that 
\begin{align}
\mathbf y 
&= \text{vec} ( \mathbf U_1^\top [\tX]_{(1)} (\mathbf U_N \otimes \mathbf U_{N-1} \otimes \ldots \otimes \mathbf U_{2}  )) =   \mathbf Z^\top \mathbf x.
\end{align}
Accordingly, 
\begin{align}
\| \mathbfcal Y\|_1  
 = \| \mathbf  y \|_1 
  \leq \sqrt{p} \|  \mathbf  y\|_2 
  =  \sqrt{p} \|  \mathbf  Z^\top \mathbf x \|_2  
  \leq 
\sqrt{p} \|  \mathbf  x \|_2   = \sqrt{p} \|  \tX \|_F.
\end{align}
\end{proof}

Lemma \ref{lemineq} shows that the L1-Tucker metric is upper bounded by $\sqrt{p} \|  \tX \|_F$.  This, in conjunction with  Proposition \ref{hooiprop}, imply that  as $q$ increases the L1-HOOI iterations converge in the L1-Tucker metric.

\begin{figure}[t!]
	\centering
	\includegraphics[width=0.7\textwidth]{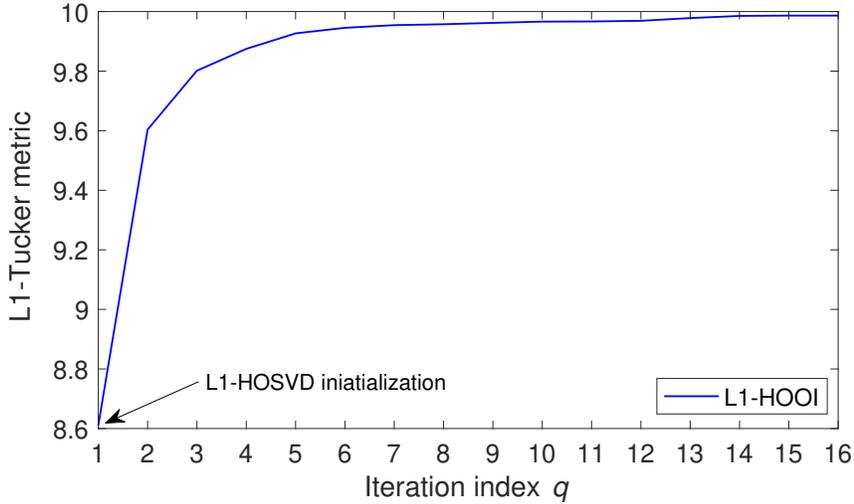} 	
	\caption{L1-Tucker metric across L1-HOOI iterations. }
	\label{synth0}
\end{figure}

To visualize the convergence, we carry out the following study. We form $5$-way tensor $\tX \in \mathbb R^{10 \times 10 \times \ldots \times 10}$, drawing independent entries from $\mathcal{N}(0,1)$. Then, we apply  L1-HOOI  on $\tX$, initialized to L1-HOSVD.   In Fig. \ref{synth0},  we plot the evolution of L1-Tucker metric $\| \tX \times_{n \in [N]} {\mathbf U_n^{(q)}}^\top  \|_1$,  versus the L1-HOOI iteration index $q$. In accordance to our formal analysis, we observe the monotonic increase of the metric and convergence after $16$ iterations.

In practice, one can terminate the L1-HOOI iterations  when the metric-increase ratio 
\begin{align}
\zeta(q) = \frac{ \left\| \tX \textstyle  \times_{n \in [N]} {\mathbf U_{n}^{(q)}}^\top \right\|_1 - \left\| \tX \textstyle  \times_{n \in [N]} {\mathbf U_{n}^{(q-1)}}^\top \right\|_1}{\left\| \tX \textstyle  \times_{n \in [N]}  {\mathbf U_{n}^{(q-1)}}^\top \right\|_1}
\end{align}
drops below a predetermined threshold $\tau>0$, or when $q$ exceeds a maximum number of permitted iterations.
Next, we discuss the computational cost of L1-HOOI.

\begin{table}[]
	\centering
	\resizebox{\columnwidth}{!}{%
		{\footnotesize
			\begin{tabular}{p{2.5 cm}p{10cm}}
				\toprule
				Method  &  Cost \\ \midrule
				\rowcolor{Gray}	{PCA (SVD)}      & $\mathcal{O}\left(D_1D_2\min\{D_1,D_2\}\right)$  \\ 
				\rowcolor{LightCyan}{L1-PCA (AO)}        &   $\mathcal{O}\left(D_1^2D_2 d_1 \right)$   \\ 
				\rowcolor{Gray}\multirow{1}{*}{{HOSVD} }	   &   $\mathcal{O}\left(\max\limits_{k \in [N]} \left\{D_k P_k  \min\{D_k, P_k\}\right\}\right)$   \\  
				\rowcolor{LightCyan}\multirow{1}{*}{{L1-HOSVD}}	   &   $\mathcal{O}\left(\max\limits_{k \in [N]} \{D_k^2 P_kd_k\}\right)$    \\  
				\rowcolor{Gray}	\multirow{1}{*}{{HOOI }}	   &  $\mathcal{O} \left(T \max\limits_{n \in [N]}\min\limits_{k \in [N]\setminus n}\left\{D_k d_k P_k {+}  D_n p_n \min\{D_n, p_n\}\right\} \right)$  \\  
				\rowcolor{LightCyan}\multirow{1}{*}{{L1-HOOI}}	    &  $\mathcal{O} \left(T \max\limits_{n \in [N]}\min\limits_{k \in [N]\setminus n}\left\{D_kd_k P_k {+} D_n^2 p_n d_n\right\} \right)$    \\  
				\bottomrule
			\end{tabular}
		}
	}
	\caption{Computational costs of PCA  \cite{GOLUB}, L1-PCA (alternating optimization with arbitrary initialization)  \cite{arlington}, HOSVD  \cite{Lathauwer}, L1-HOSVD (proposed), HOOI \cite{Kolda}, and  L1-HOOI (proposed).   PCA/L1-PCA costs are reported for input matrix $\mathbf X \in \mathbb R^{D_1 \times D_2}$ and decomposition rank $d_1$.  Tucker/L1-Tucker costs are reported for input tensor $\tX \in \mathbb R^{D_1 \times D_2 \times \ldots \times D_N}$ and mode-$n\in [N]$ ranks $\{d_n\}_{n \in [N]}$, for $P_k=\prod_{n \in [N]\setminus k}D_n$ and $p_k =\prod_{n \in [N]\setminus k}d_n.$ $T$ is the maximum number of  iterations conducted by HOOI and L1-HOOI.}
	\label{complexities}
\end{table}

\subsubsection{Complexity Analysis}

As studied above, initialization of L1-HOOI by means of L1-HOSVD costs  $\mathcal O(\max_{k \in [N]} D_{k} d_k  P )$, where $P=\prod_{m \in [N]} D_m$.
Then, at iteration $q$, L1-HOOI computes matrix $\mathbf A_{n}^{(q)}$ in \eqref{l1hooi2} and  its L1-PCA,  for every $n$.  Matrix $\mathbf A_{n}^{(q)}$ can computed by a sequence of matrix-to-matrix products as follows. First, we compute the mode-$k$ product of $\tX$ with $\mathbf U_{k}^{(z_k)}$, for some $k \neq n$ ($z_k=q$ if $k<n$ and $z_k=q-1$ if $k>n$), with cost $\mathcal O(d_k P)$. Next, we compute the $l$-mode product of $\tX \times_k \mathbf U_{k}^{(z_k)}$ with $\mathbf U_l^{(z_l)}$, for some $l \notin \{n,k\}$,  with cost $\mathcal O(d_l d_k P_k)$, where $P_k=\prod_{m \in [N]\setminus k} D_m$. We observe that the second  product has lower cost than the first, for any selection of $k$ and $l$. Similarly, each subsequent mode product will have further  reduced cost. Thus, keeping the dominant term (cost of first product) and taking products in a computationally favorable order, the computation of $\mathbf A_{n}^{(q)}$ costs $\mathcal O(\min_{k \in [N]\setminus n}d_k P)$. Importantly, the cost of computing $\mathbf A_{n}^{(q)}$ is the same for every iteration index $q$. After $\mathbf A_{n}^{(q)}$  is computed, \eqref{l1hooi2} is solved with cost $\mathcal O(D_n^2 p)$, as shown above, where $p=\prod_{m \in [N]} d_m$. Thus, for fixed $(q,n)$, computation and L1-PCA of $\mathbf A_{n}^{(q)}$ cost $\min_{k \in [N]\setminus n}\{d_k P {+} D_n^2 p\}$. Then, we observe that, there exists mode index $n \in [N]$ such that, the cost of computing $\mathbf A_{n}^{(q)}$ and its L1-PCA basis dominates over the respective cost of the $m$-mode for every $m \in [N]\setminus n$. The latter implies that the dominant cost at iteration $q$ of L1-HOOI is $\max_{n \in [N]}\min_{k \in [N]\setminus n}\{d_k P + D_n^2 p \}$. Denoting by  $T$   the maximum number of iterations permitted by L1-HOOI,  the overall computational cost of L1-HOOI  is $\mathcal{O}(T (\max_{n \in [N]}\min_{k \in [N]\setminus n}\{D_k d_k P_k +D_n^2 p\})$. In Table \ref{complexities}, we offer  the computational costs of PCA, L1-PCA, HOSVD, L1-HOSVD, HOOI, and L1-HOOI.  We observe that the costs of the proposed L1-HOSVD and L1-HOOI algorithms are comparable with those of standard HOSVD and HOOI respectively.
Next, we conduct numerical studies and compare the performance of standard Tucker solvers with that of the proposed L1-Tucker algorithms.

\section{Numerical Studies}
 
\subsection{Tensor Reconstruction}
 \label{recstud}
 
We set $N=5$,  $D_1=D_3=D_5=10$, $D_2=D_4=15$, $d_1=d_2=6$, $ d_3=d_4=d_5=4$, and generate Tucker-structured $\tX=\tG \times_{n \in [5]} \mathbf U_n$. The core tensor $\tG$ draws entries from $\mathcal N(0,9)$ and, for every $n$, $\mathbf{U}_n$  is an arbitrary basis. We corrupt all entries of $\tX$ with zero-mean unit-variance additive white Gaussian noise (AWGN), disrupting its Tucker structure. Moreover, we corrupt $N_o$ out of the $P=\prod_{i=1}^5 D_i = 225,000$ entries of $\tX$ --corruption ratio $\rho = \frac{N_o}{P}$-- by adding high magnitude outliers drawn from $\mathcal{N}(0,\sigma_o^2)$.  Thus, we form $\tX^{\text{corr}}=\tX+\mathbfcal N +\tO$, where $\mathbfcal N$ and $\tO$ model AWGN and sparse outliers, respectively.
Our objective is to reconstruct $\tX$ from the available $\tX^{\text{corr}}$. Towards our objective, we Tucker decompose $\tX^{\text{corr}}$ by means of HOSVD, HOOI,    L1-HOSVD, and L1-HOOI and obtain bases $\{\hat{\mathbf U_n}\}_{n\in [5] }$.  Then, we reconstruct $\tX$ as $\hat{\tX}=\tX^{\text{corr}}\times_{n \in [5]}\hat{\mathbf U_n}\hat{\mathbf U_n^\top}$. The   normalized squared error  (NSE) is defined as
$
{\| \tX - \hat{\tX} \|_F^2}{ \| \tX \|_F^{2}}^{-1}.
$
In Figure \ref{synth2fig}, we set $N_o=300$ ($\rho = 1.33~10^{-3}$) and plot the mean NSE (MNSE), evaluated over 1000 independent noise/outlier realizations, versus outlier standard deviation $\sigma_o=4,8,\ldots,28$. In the absence of outliers ($\sigma_o=0$), all methods under comparison exhibit similarly low  MNSE. As the outlier standard deviation $\sigma_o$ increases the MNSE of all methods increases. We notice that the performances of HOSVD and HOOI  markedly deteriorate for $\sigma_o\geq 12$ and $\sigma_o\geq 20$ respectively.
On the other hand, L1-HOSVD and L1-HOOI remain robust against corruption, across the board.

 \begin{figure}[t!]
	\centering
	\begin{subfigure}{0.495\linewidth}
	\includegraphics[width=1\linewidth]{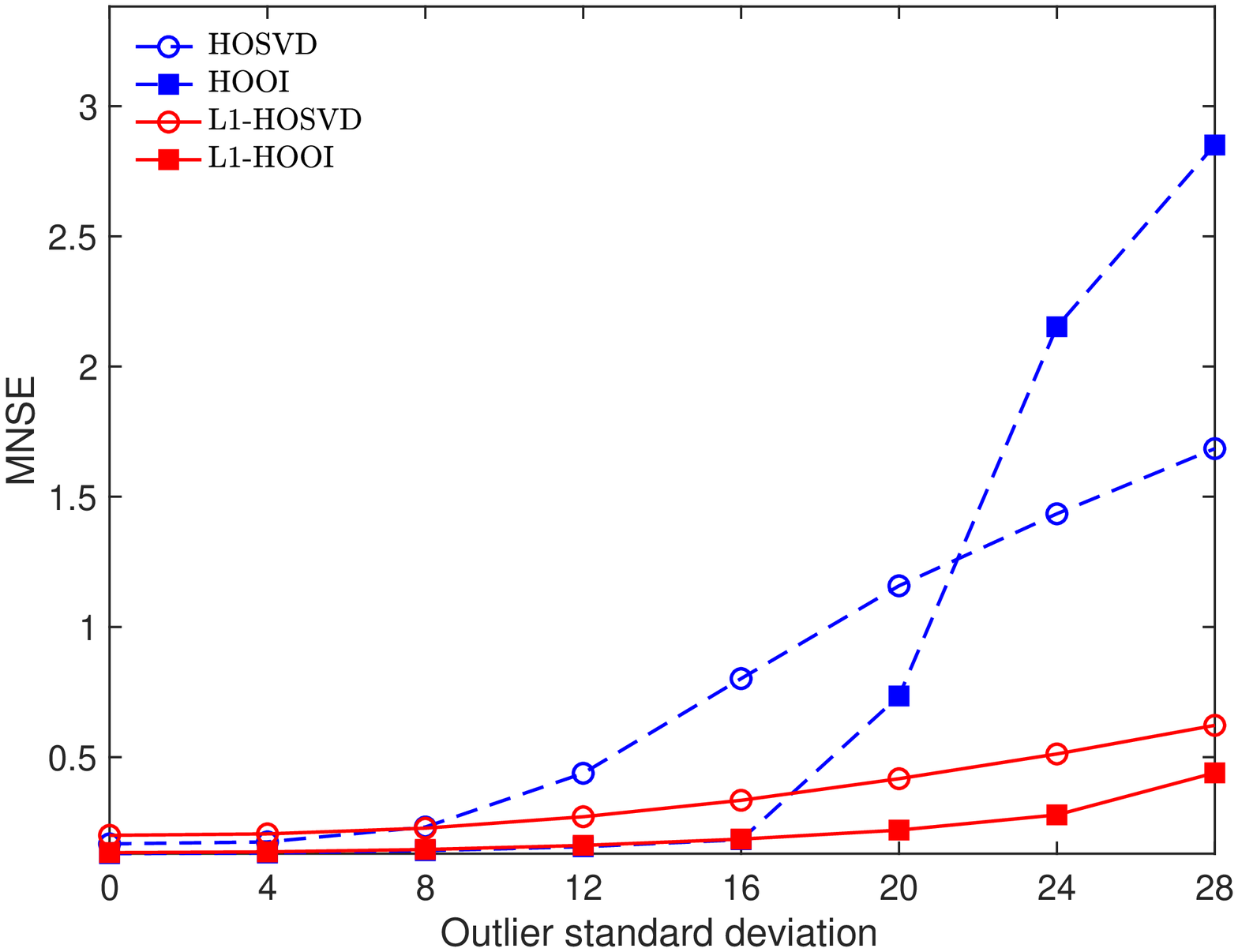} 
	\caption{}
	\label{synth2fig}
	\end{subfigure}~
    \begin{subfigure}{0.495\linewidth}	
	\includegraphics[width=1\linewidth]{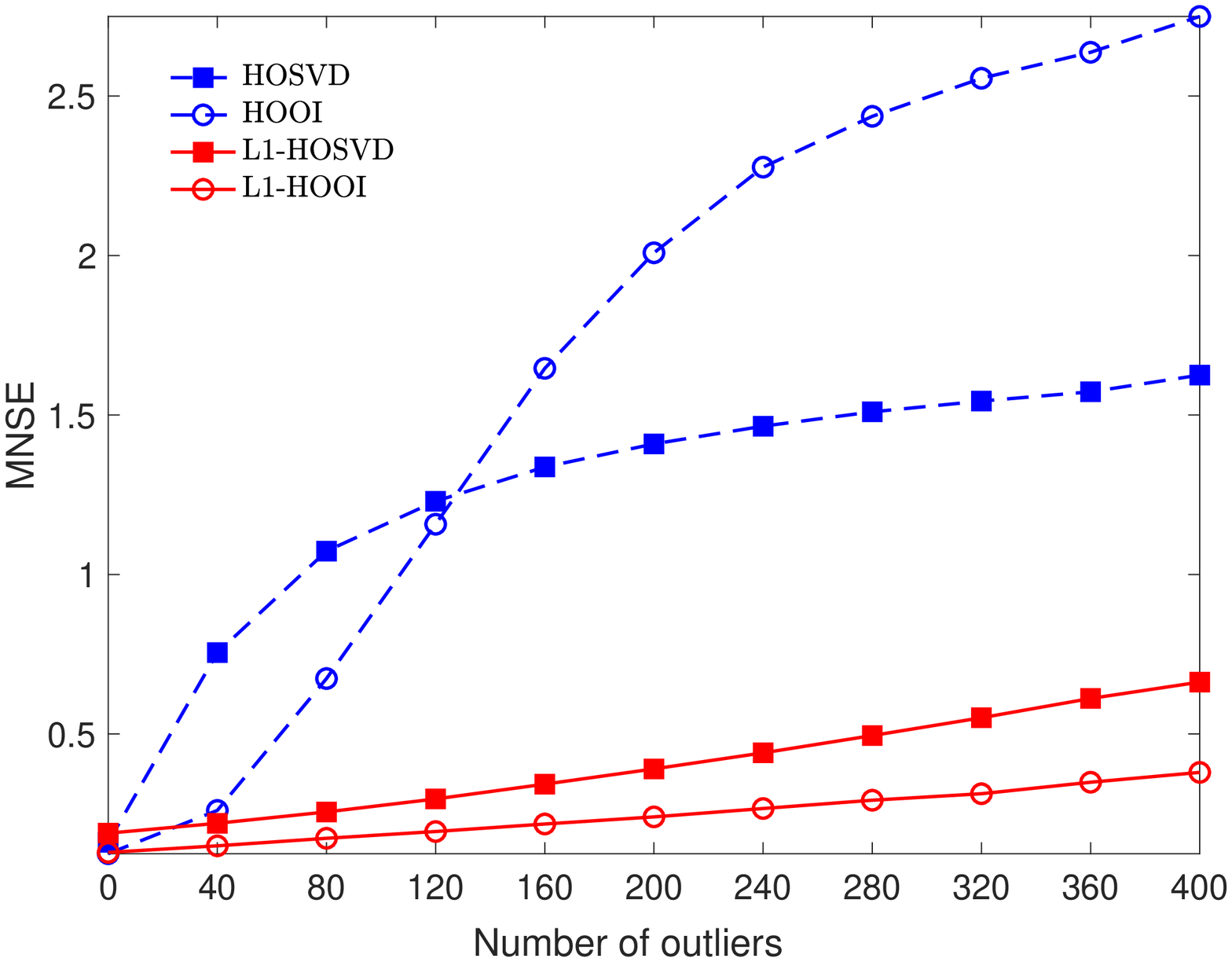}
	\caption{}
	\label{synth1fig} 
	\end{subfigure}
~
 \begin{subfigure}{0.7\textwidth}
\includegraphics[width=1\textwidth]{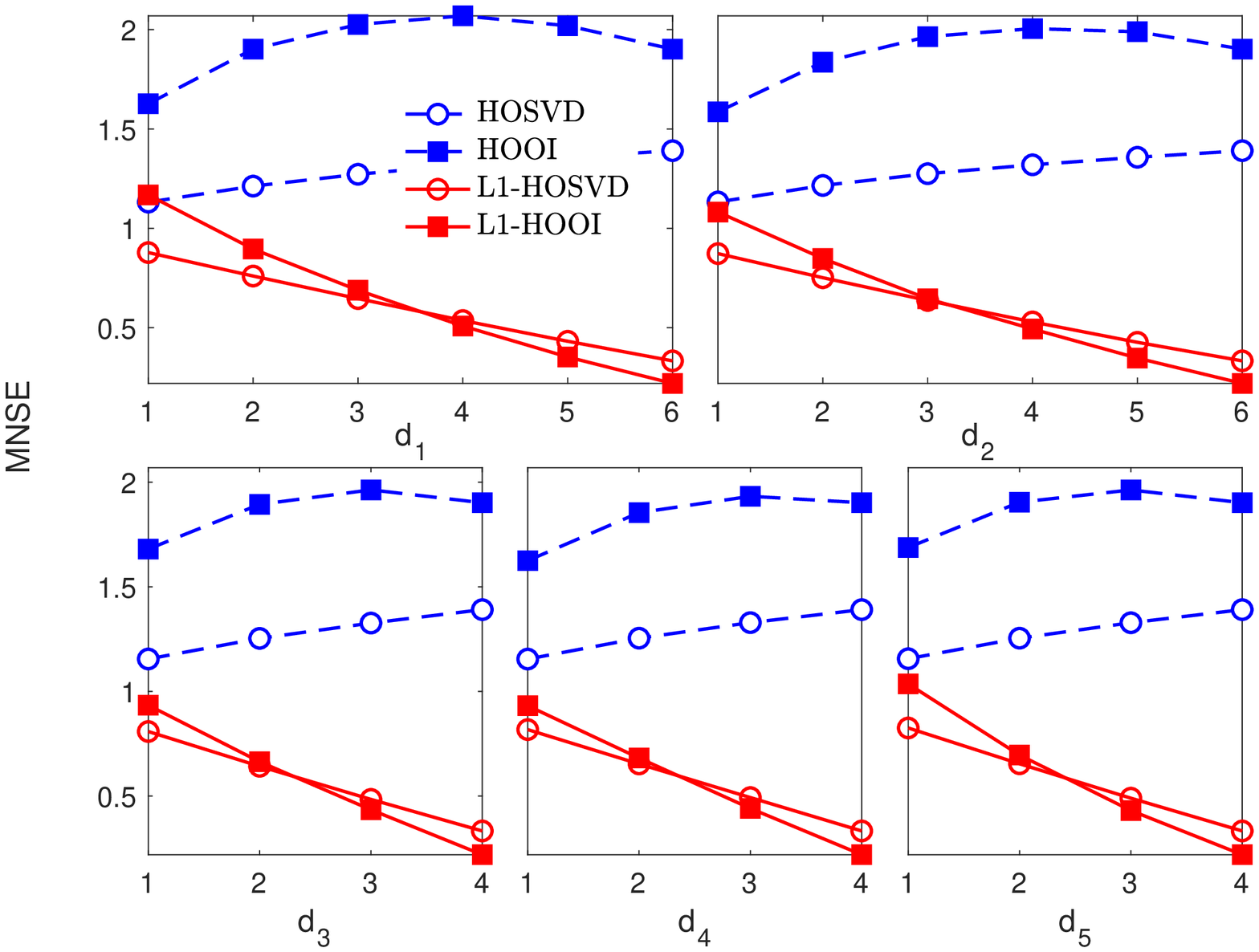} 
\caption{}	
\label{synth3fig} 
\end{subfigure}%
	\caption{(a) MNSE versus  standard deviation $\sigma_o$ for $N_o=300$. (b) MNSE versus number of outliers $N_o$ for  $\sigma_o=26$. (c) MNSE versus ${d_n}_{n\in [N=5]}$ for $N_o=150$  and $\sigma_o=28$; for every $m$,  $d_m$ is set to its nominal value and $d_n$ variance,  $n \in [5]\setminus m$.}
	\label{synth2}
\end{figure}

In Figure \ref{synth1fig}, we set $\sigma_o=26$ and plot the MNSE versus number of outliers $N_o=0,40,\ldots,400$. Expectedly,  in  the absence of outliers ($N_o=0$),  all methods exhibit low MNSE. As the number of outliers increases, HOSVD and HOOI  start exhibiting high reconstruction error, while L1-HOSVD and L1-HOOI remain robust. For instance,  the MNSE of L1-HOSVD for $N_o=400$  outliers is lower than the MNSE of standard HOSVD for $N_o=40$ (ten times fewer) outliers.

Finally, in Figure \ref{synth3fig}, we set $\sigma_o=28$, $N_o=150$ ($\approx0.07 \%$ of total data entries are corrupted) and plot the MNSE versus $d_n \forall n$ while $d_m$ is set to its nominal value for every $m \in [N=5]\setminus n$. We observe that, even for a very small fraction of outlier corrupted entries in $\tX^{\text{corr}}$, standard Tucker methods are clearly misled across all $5$ modes. On the other hand, the proposed L1-Tucker counterparts, exhibit sturdy outlier resistance and reconstruct $\tX$ well, remaining almost unaffected by the outlying entries in $\tX^{\text{corr}}$.

A robust tensor analysis algorithm, specifically designed for counteracting sparse outliers, is the High-Order Robust PCA (HORPCA) \cite{goldfarb}. Formally, given $\tX^{\text{corr}}$, HORPCA solves 
\begin{align}
\begin{matrix}
\underset{\hat{\tX}, ~\hat{\mathbfcal O} \in \mathbb R^{D_1 \times \ldots \times D_N}}{\text{minimize}}~ &\sum_{n=1}^{N} \| [\hat{\tX}]_{(n)} \|_* + \lambda \|  \hat{\mathbfcal O} \|_1 \\
\text{subject to}~& \hat{\tX} +  \hat{\mathbfcal O} = \tX^{\text{corr}}
\end{matrix}.
\label{horpca}
\end{align} 
Authors in \cite{goldfarb} presented the HoRPCA-S algorithm for the solution of \eqref{horpca} which relies on a specific sparsity penalty parameter $\lambda$, as well as a thresholding variable $\mu$. The model in \eqref{horpca}  was introduced considering that, apart from the sparse outliers, there is no dense (full rank) corruption to $\tX$ (see \cite{goldfarb}, Section 2.6). In the case of additional dense corruption, HORPCA is typically accompanied by HOSVD \cite{goldfarb, li2015low,lu2016tensor}. In the sequel, we refer to this approach as HORPCA+HOSVD.

\begin{figure}[t!]	
	\centering
	\includegraphics[width=0.6\textwidth]{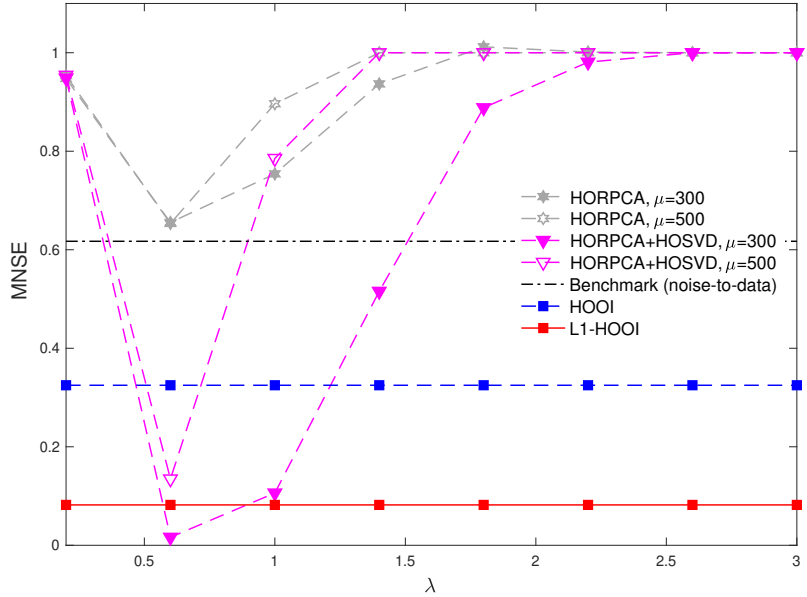}	
	\caption{MNSE versus $\lambda$ for varying  $\mu$.}
	\label{mesh}
\end{figure}

In our next  study, we set $N=5$, $D_n=5$, and $d_n=2$ for every $n$, and build the Tucker-structured  data tensor $\tX=\tG \times_{n \in [5]} \mathbf U_n$, where the entries of core $\tG$ are independently drawn from $\mathcal N(0,12^2)$. Then, we add both dense AWGN and sparse outliers, creating $\tX^{\text{corr}}=\tX + \mathbfcal N + \mathbfcal O$, where the entries of noise $\mathbfcal N$ are drawn independently from $\mathcal N(0, 1)$ and the  $15$ non-zero  entries of  $\mathbfcal O$ (in arbitrary locations) are drawn from $\mathcal N(0,20^2)$.  Then, we  attempt to reconstruct $\tX$ from the available $\tX^{\text{corr}}$ using  HOOI, HORPCA (for $\lambda=0.2, 0.6, \ldots, 3$ and $\mu=300,500$), HORPCA+HOSVD (same  $\lambda$ and $\mu$ combinations as HORPCA), and the proposed L1-HOOI. 

In Figure \ref{mesh}, we plot  MNSE computed  over $50$ data/noise/corruption realizations, 
versus $\lambda$ for the four methods. In addition, we plot the average noise-to-data  benchmark  
$
{\| \mathbfcal N \|_F^2}{ \| \tX \|_F^{-2}}.
$
In accordance with our previous studies, we observe that L1-HOOI offers markedly lower MNSE than standard HOOI. In addition, we notice that for   specific selection of $\mu$ and $\lambda$  ($\mu=300$ and $\lambda=0.6$) HORPCA+HOSVD may attain MNSE even lower than HOOI. However,  for any  different  selection of $\lambda$,  HORPCA+HOSVD attains higher MNSE than HOOI. In addition, we plot the performance of HORPCA when it is not followed by HOSVD. We notice that, expectedly, for specific selections of $\mu$ and $\lambda$ the method is capable of removing the outliers, but not the dense noise component --thus,  the MNSE approaches the average noise-to-data benchmark. This study highlights the corruption-resistance of L1-HOOI, while, similarly to HOSVD and HOOI, it does not depend on any tunable parameters, other than $\{d_{n}\}_{n \in [N]}$.

\subsection{Classification} 
 \label{classtud}

{Tucker decomposition is commonly employed for classification of multi-way data samples. Below, we consider the Tucker-based classification framework originally presented in \cite{Phan10}.}  That is,  we consider $C$ classes of  order-$N$ tensor objects of   size $D_1 \times D_1 \times \ldots \times D_N$ and $M_c$ labeled samples available from  the $c$-th class, $c \in [C]$, that can be used for training a classifier.  
The training data from class $c$ are organized in tensor $\tX_{c} \in \mathbb R^{D_1 \times D_2 \times \ldots \times D_N \times M_c}$ and the total of  $M=\sum_{c=1}^C M_c$ training data are organized in tensor $\tX \in \mathbb R^{D_1 \times \ldots \times D_N \times M}$, constructed by concatenation of $\tX_1, \ldots, \tX_C$ across mode $(N+1)$.

In the first processing step,   $\tX$ is Tucker decomposed, obtaining  the feature bases $\{\mathbf U_n \in \mathbb S(D_n, d_n) \}_{n \in [N]}$ for the first $N$  modes  (feature modes) and the sample basis $\mathbf Q \in \mathbb S(M,M)$ for the $(N+1)$-th   mode (sample mode). The obtained feature bases are then  used to compress the training data, as
\begin{align}
\mathbfcal G_c= \tX_c \times_{n \in [N]} \mathbf U_n^\top \in \mathbb R^{d_1 \times \ldots \times d_N \times M_c}
\end{align}
for every $c \in [C]$.
Then, the $M_c$ compressed tensor objects from the $c$-th class are vectorized (equivalent to mode-$(N+1)$ flattening) and stored in the data matrix 
\begin{align}
\mathbf G_c = [\mathbfcal G_c]_{(N+1)}^\top \in \mathbb R^{p \times M_c},
\end{align}
where $p= \prod_{n \in [N]} d_n$. Finally, the labeled columns of $\{\mathbf G_c\}_{c \in [C]}$ are used to train any standard vector-based classifier, such as support vector machines (SVM), or $k$-nearest-neighbors ($k$-NN). 

When an unlabeled testing point $\mathbfcal Y \in \mathbb R^{D_1 \times \ldots \times D_N}$ is received, it is first compressed using the Tucker-trained bases as $\mathbfcal Z = \mathbfcal Y \times_{n \in [N]} \mathbf U_n^\top$. Then,  $\mathbfcal Z$ is vectorized  as $\mathbf z = \text{vec}(\mathbfcal Z) = [\mathbfcal Z]_{(N+1)}^\top$. Finally, vector $\mathbf z$ is classified based on the standard vector classifier trained above.

In this study, we focus on the classification  of order-$2$ data ($N=2$) from the MNIST  image dataset of handwritten digits \cite{mnist}.  Specifically, we consider  $C=5$ digit classes (digits $0, 1, \ldots, 4$) and   $M_1 = \ldots = M_5=10$  image samples of size $(D=D_1=28) \times (D=D_2)$ available from each class.   
To make the classification task more challenging, we consider that each training image is corrupted by heavy-tail noise with probability $\alpha$. Then, each pixel of a corrupted image is additively corrupted by heavy tailed noise $n \sim \text{unif}(0,v)$, with probability $\beta$. Denoting the average pixel energy by $w^2=\frac{1}{D^2 M} \| \tX \|_F^2$, we choose $v$ so that $ \frac{w}{ \sqrt{\mathbb E\{ n^2\}}} =  10 $.  

We conduct Tucker-based classification as  described above, for $d=d_1=d_2$, using a  nearest-neighbor (NN)  classifier (i.e., $1$-NN), by which testing sample    $\mathbf z$ is assigned to class\footnote{We consider a simple classifier, so that the  study focuses to the impact of each compression method.} 
\begin{align}
c^* = \underset{c \in [C]}{\text{argmin}} \left\{ \underset{ j \in [M_c]}{\min}   \| \mathbf z - [\mathbf G_c]_{:,j} \|_2^2 \right\}.
\end{align}

 \begin{figure}[t!]
	\centering
	\includegraphics[width=0.6\textwidth]{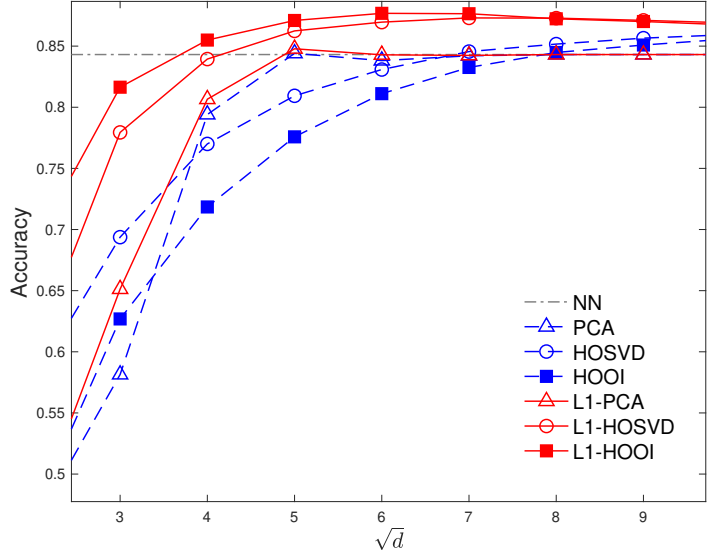}
	\caption{Classification accuracy versus ${d}$, for $\alpha=0.2$ and $\beta=0.5$. }
	\label{acc1}
\end{figure}

For a given training dataset, we  classify $500$ testing points from each class. Then, we repeat the training/classification procedure on $300$ distinct  realizations of training data, testing data, and corruptions.
In Figure \ref{acc1}, we plot the average  classification accuracy versus ${d}$ for $\alpha=0.2$ and $\beta=0.5$,  for HOSVD, HOOI, L1-HOSVD, L1-HOOI, as well as PCA, L1-PCA,\footnote{Denoting by $\mathbf U$ the $\min\{p,M\}$ PCs/L1-PCs of $[\tX]_{(N+1)}^\top \in \mathbb R^{P \times M}$, we train any classifier on the labeled columns of $\mathbf U^\top [\tX]_{(N+1)}$ and classify the vectorized and projected testing sample $\mathbf U^\top \text{vec}(\mathbfcal Y)$.} and plain NN classifier that  returns the label of the nearest column of $[\tX]_{(N+1)}^\top \in \mathbb R^{P \times M}$ to the vectorized testing sample  $\text{vec}(\mathbfcal Y)$.
We observe that, in general, the compression-based methods  can attain superior performance than plain NN. Moreover, we notice that   ${d}>7$ implies $p>M$ and, thus,  the PCA/L1-PCA  methods attain constant performance, equal to plain NN. Moreover, we notice that L1-PCA outperforms PCA, for every value of $d \leq 7$. For $4 \leq d \leq 7$, PCA/L1-PCA outperform the Tucker methods. 
Finally, the proposed L1-Tucker methods outperform standard Tucker and PCA/L1-PCA, for every $d$, and attain the highest classification accuracy of about $89\%$ for $d=6$ ($5\%$ higher than plain NN).  

 \begin{figure}[t!]
	\centering
		\includegraphics[width=0.6\textwidth]{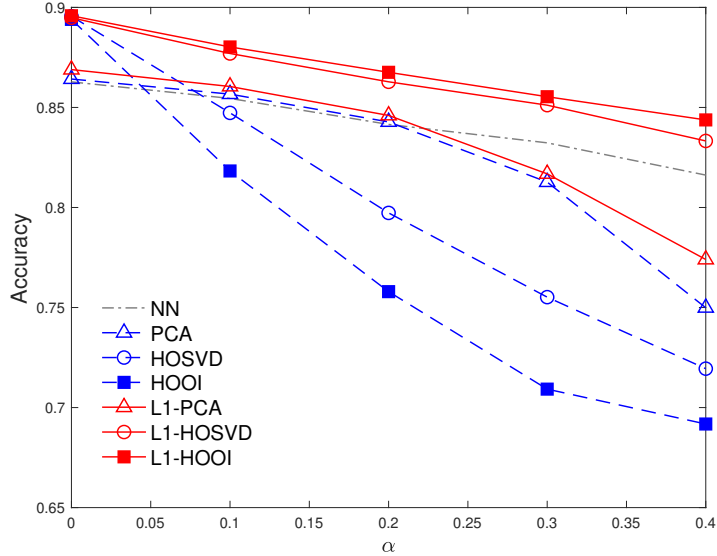} 
	\caption{Classification accuracy versus $\alpha$, for $d=5$ and $\beta=0.8$.}
	\label{acc2}
\end{figure}

 \begin{figure}[t!]
	\centering
				\includegraphics[width=0.6\textwidth]{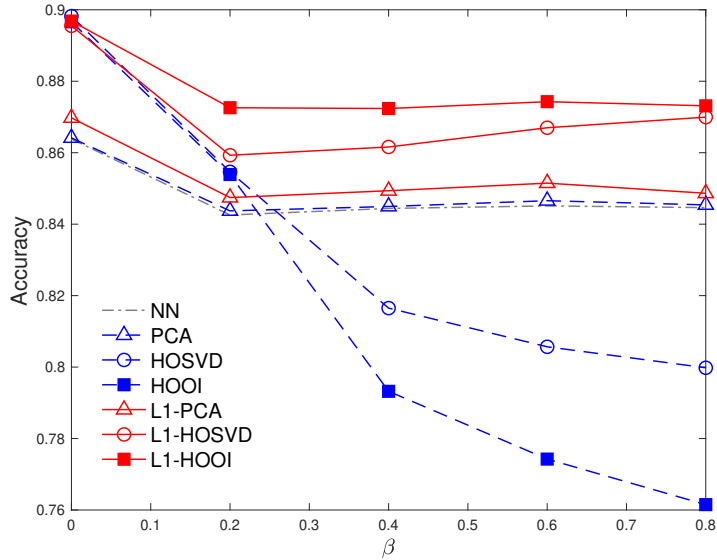}
	\caption{Classification accuracy versus $\beta$, for $\alpha=0.2$ and $d=5$.}
	\label{acc3}
\end{figure}

Next, we fix  $d=5$ and $\beta=0.8$ and plot in Figure \ref{acc2} the average classification accuracy, versus $\alpha$.  This figure reveals the sensitivity of standard HOSVD and HOOI as the training data corruption probability increases. At the same time,  the proposed L1-Tucker methods exhibit   robustness against the corruption, maintaining the highest average accuracy for every value of $\alpha$. For instance, for image-corruption probability $\alpha=0.3$, L1-HOSVD and L1-HOOI attain about $ 87\%$ accuracy, while HOSVD and HOOI attain accuracy $ 75\%$ and $ 71\%$, respectively.

Last, in  Figure \ref{acc3}, we plot the average classification accuracy, versus the pixel corruption probability $\beta$, fixing again $\alpha=0.2$ and $d=5$. We observe that, for any value of $\beta$, the performance of  the L1-HOSVD and L1-HOOI does not drop below $86\%$ and $87.5\%$, respectively. On the other hand, as $\beta$ increases, NN and PCA-based methods perform close to $85\%$. The performance of standard Tucker methods decreases markedly, even as low as $76\%$, for intense corruption with $\beta=0.8$. 
The above studies highlight the benefit of L1-Tucker compared to standard Tucker and PCA counterparts. 
 
\section{Conclusions}

We studied L1-Tucker,  an L1-norm based reformulation of standard Tucker tensor decomposition. In addition, we presented two algorithms for its solution, L1-HOSVD and L1-HOOI. Both algorithms were accompanied by formal complexity and convergence analysis. We carried out numerical studies on tensor reconstruction and classification,  both on synthetic and on real data, comparing the proposed L1-Tucker methods with standard counterparts. In our numerical studies, L1-Tucker performed similar to standard Tucker  when the processed data are corruption-free, while, in contrast to Tucker, it attained sturdy resistance against heavy corruptions.


\ifCLASSOPTIONcaptionsoff
  \newpage
\fi



%
\bibliographystyle{siamplain}
\bibliography{l1tucker_ref}

%

\end{document}